\documentclass{amsart}
\usepackage[latin1]{inputenc}
\usepackage{amsmath, amssymb, amsthm}
\usepackage{amsfonts}
\usepackage{pst-tree}

\def \P{\mathcal{P}}
\def \l{\ell}
\def \Psq{\mathcal{P}_{\rm sq}}
\def \Z{\mathbb{Z}}
\def \N{\mathbb{N}}

\def \T{\mathcal{T}}
\def \tmu{\tilde\mu}
\def \({\left( }
\def \){\right)}
\def \t{\tilde}

\theoremstyle{definition}
\newtheorem{defi}{Definition}
\newtheorem{lemma}[defi]{Lemma}
\newtheorem{claim}[defi]{Claim}
\newtheorem{thm}[defi]{Theorem}
\newtheorem{rem}[defi]{Remark}
\newtheorem{example}[defi]{Example}

\title{Partitions, pairs of trees and Catalan numbers}
\author{Eliana Zoque}
\address{University of California, Riverside}\email{elizoque@math.ucr.edu}
\date{\today}
\begin{document}
\begin{abstract}
Bennett et al.\ \cite{BCDM} presented a recursive algorithm to
create a family of partitions from one or several partitions. They
were mainly interested in the cases when we begin with a single
square partition or with several partitions with only one part.
The cardinalities of those families of partitions are the Catalan
and ballot numbers, respectively. In this paper we present a
closed description for those families.  We also present bijections
between those sets of partitions and sets of trees and forests
enumerated by the Catalan an ballot numbers.
\end{abstract}

\maketitle

\section{Introduction}
The Catalan numbers appear in a wide variety of settings,
including representation theory. While studying the category of
finite dimensional representations of the affine Lie algebra
associated to $\mathfrak{sl}_2$ and trying to develop a theory of highest
weight categories, Chari and Greenstein (\cite{CG1,CG2}) found that that
one of the results required for this would be to prove that a
certain module for the ring of symmetric functions is free of rank
equal to the Catalan number. This module is generated by
polynomials that are described using a family of partitions
defined recursively by Bennett et al.\ \cite{BCDM}.

Their algorithm consist of applying two operations to a set of
partitions: the first operation is augmentation which increases
the partition by one part, and to the partitions so obtained we
apply a involution (these operations will be defined in Section
\ref{notation}). If we start with a square partition, i.e.
$\lambda=(\lambda_1,\lambda_2,\dots,\lambda_k)$ where
$k=\lambda_1\geq\lambda_2\geq\dots\geq\lambda_k=1$ then after
applying the algorithm $\l-k$ times we get a set of $c_{\l-k+1}$
partitions with exactly $\l$ parts, where
$c_n=\frac1{n+1}{{2n}\choose{n}}$ is the Catalan number. If we
start with the set of partitions with only one part less than or
equal than $m$, then after $\l-1$ steps we get a set of
$b_{\l,m-1}$ partitions with exactly $\l$ parts, where
$$b_{\ell,m}=\binom{2\ell+m}{\ell}-\binom{2\ell+m}{\ell-1},$$
is called a ballot number.

The main goal of the present paper is to study those families of partitions from a combinatorial point of view. To do that, we first present closed
descriptions for those families of partitions using inequalities, this is done
in Section \ref{ineq}. In Section \ref{bijections} we provide
bijections to families of trees and forests that are enumerated by
the Catalan and ballot numbers, respectively.

\subsection*{Acknowledgements} The author is grateful to M.\ Bennett, V. Chari, R.J.\ Dolbin and N.
Manning for posing the questions and for useful discussions and suggestions.

\section{Notation}\label{notation}

In this section we will present the relevant notation, as well as
some Theorems taken from \cite{BCDM}.

Let $[n]=\{1,\dots, n\}$ for any positive integer $n$. By a
partition $\lambda$ with $n$ parts, we  mean a decreasing sequence
$$\lambda=(\lambda_1,\lambda_2,\cdots,\lambda_n),\quad \lambda_1\geq\lambda_2\geq\dots\geq\lambda_k>0$$ of positive integers.

We denote the set of all partitions by $\P$. Given $\lambda=(\lambda_1,\lambda_2,\cdots,\lambda_n)\in \P$ set
$$\lambda\setminus\{\lambda_n\}=
(\lambda_1,\cdots,\lambda_{n-1}),$$ and for $0 < \lambda_{n+1}
\leq \lambda_n$ set
$$(\lambda:\lambda_{n+1})=(\lambda_1,\lambda_2,\cdots,\lambda_n,\lambda_{n+1}).$$

For $k, n\in\Z^+$, let $\P^{n,k}$ be the set of partitions with
exactly  $n$ parts where no part is bigger than $k$, i.e.
$$\P^{n,k}=\{\lambda=(\lambda_1,\lambda_2,\cdots,\lambda_n)\in\P:\lambda_1\le k\}.$$
We can regard a partition $\lambda\in\P^{n,k}$ as a decreasing
function $\lambda:[n]\to[k],\,i\mapsto \lambda_i$, so we can write
$\lambda(i)=\lambda_i.$

Let $\tau:\P^{n,k}\to\P^{n,k}$ (or $\tau_{k}$, if confusion arises) be
defined by
$$\tau(\lambda_1,\cdots,
\lambda_n)=(k+1-\lambda_n,\cdots, k+1-\lambda_1).$$ As a map
$[n]\to[k]$, $\tau(\lambda)$ is equal to the composition
$\gamma_{k}\circ\lambda\circ\gamma_n$, where for every positive
integer $m$, $\gamma_m:[m]\to [m]$ is a bijection of order two
defined by $i\mapsto m+1-i$.

Set $\P^k=\P^{k,k}$.  Given $\lambda\in\P^k$ and $\l,k\in\Z^+$
with $\l\geq k$, define subsets $\P^\l(\lambda)$ of $\P^\l$
inductively, by
\begin{gather*}\P^k(\lambda)=\{\lambda\}\cup\{\tau_k\(\lambda\)\},\quad
\P^\l(\lambda)=\P^\l_{\rm
d}(\lambda)\cup\P^\l_\tau(\lambda),\end{gather*}
where
\begin{gather*}\P_{\rm d}^\l(\lambda)=\{\mu\in\P^\l:
\mu\setminus\{\mu_\l\}\in\P^{\l-1}(\lambda)\},\quad
\P^\l_\tau(\lambda)=\tau_\l\(\P^\l_{\rm
d}(\lambda)\).\end{gather*} It is easy to see that
\begin{equation*}\label{taustab}\P^\l\(\tau_k\(\lambda\)\)=
\tau_\l\(\P^\l(\lambda)\)=\P^\l(\lambda).\end{equation*}

For $k\in\Z^+$, set \begin{gather*}
\Psq^k=\{\lambda\in\P^k:\lambda_1=k,\ \
\lambda_k=1\}.\end{gather*}

Fix $m\in\Z^+$, and let
$$\Omega_m=\{\{j\}: 1\le j\le m\}\subset \P^{1,m}.$$
Define subsets $\P^{\l}(\Omega_m)$ by\begin{gather*}\P_{\rm
d}^1(\Omega_m)=\Omega_m=\P^1_{\tau}(\Omega_m)=\tau_m\(\Omega_m\),\\
\P^\l_{\rm d}(\Omega_m)=\{(\mu:j)\in\P^{\l,\l+m-1}\,:\,1\leq j\leq
\mu_{\l-1},\,\mu\in\P^{\l-1}(\Omega_m)\},\\
\P^\l_\tau(\Omega_m)=\tau_{\l+m-1}\(\P^\l_{\rm d}(\Omega_m)\),\\
\P^\l(\Omega_m)=\P^\l_{\rm d}(\Omega_m)\cup\P^\l_\tau(\Omega_m).
\end{gather*}

The following are the most important theorems of \cite{BCDM}. We
will present alternative proofs of these theorems in the upcoming
sections.

\begin{thm}[\cite{BCDM}, Section 1.5]\label{thm1} \ \begin{enumerate}
\item[(i)] Let $\l, k\in\Z^+$ be such that $\l\geq k$ and let
$\lambda\in\Psq^k$. Then, $$\#\P^\l(\lambda)=\begin{cases}
c_{\l-k+1},&\lambda=\tau(\lambda),\\2c_{\l-k+1},&\lambda\neq
\tau(\lambda).\end{cases}$$ \item[(ii)] Let $\lambda\in\Psq^k$,
$\nu\in\Psq^s$. For all $\l\in\Z^+$ with $\l\geq\max(k,s)$, we
have
$$\P^\l(\lambda)\cap\P^\l(\nu)=\emptyset,\text{ if }
\nu\notin\{\lambda,\tau_k(\lambda)\}.$$ \item[(iii)] We have
$$\P^\l =\bigcup_{\{\lambda\in\Psq^k\,:\, \l\geq
k\geq1\}}\P^\l(\lambda).$$
\end{enumerate}
\end{thm}

\begin{thm}[\cite{BCDM}, Section 3.1]\label{thm2} For $\l,m\in\Z^+$, we have  $$\#\P^\l(\Omega_m)=b_{\l,m-1}.$$
\end{thm}

\section{A closed characterization using inequalities}\label{ineq}

In this section we provide a closed description of the elements of
of $\P^\l(\lambda)$ and $\P^\l(\Omega_m)$ using inequalities.

\subsection{Inequalities for $\boldsymbol{\P^\l(\lambda)}$}

\begin{thm}\label{sq}
Let $\l\geq k$, $\lambda\in\Psq^k$ and $\mu\in\P^\l$. Then $\mu\in\P^\l(\lambda)$ if and only if there
exists $b\in[\l-k+1]$ so that:
\begin{enumerate}
\item[(i)] $\mu(b)=b+k-1,\,\mu(b+k-1)=b$. \item[(ii)]
 If $1\leq i< b$ then $\mu(\mu(i))>i$. \item[(iii)] If
$b+k-1< i\leq \l$ then $\mu(\mu(i))<i$.
\item[(iv)] If $\t\lambda\in\Psq^k$ is defined by
$\t\lambda(i)=\mu(b+i-1)-b+1,\,1\leq i\leq k$ then $\t\lambda\in\{\lambda,\tau(\lambda)\}.$
\end{enumerate}
\end{thm}

\begin{rem}
It is clear that, given $\mu\in\P^\l$, if there exist $b$ and $k$
satisfying conditions (i)-(iii) they are unique since $b$ and
$b+k-1$ are the smallest and largest fixed points of
$\mu\circ\mu:[\l]\to[\l]$. We will prove in Lemma \ref{existence}
that they do exist for every $\mu\in\P^\l$.

Also, notice that Theorem \ref{thm1}(ii) is a direct consequence
of condition (iv) in Theorem \ref{sq}.
\end{rem}

Before proving Theorem \ref{sq} we prove the following
\begin{lemma}\label{ii-iii}\ \begin{enumerate}\item
$\mu\in\P^\l$ satisfies (i), (ii) and (iii) for $b$ if and only if
$\tau_\l(\mu)$ satisfies (i), (ii) and (iii) for $b'=\l-k-b+2.$
\item If  $\t\lambda\in\Psq^k$ is defined as in (iv), i.e.,
$$\t\lambda(i)=\mu(b+i-1)-b+1,\,1\leq i\leq k$$ and $\t{\lambda'}\in\Psq^k$ is defined similarly for $\tau_\l(\mu)$ and $b'$,
i.e., $$\t{\lambda'}(i)=\tau_\l(\mu)(b'+i-1)-b'+1,\,1\leq i\leq k$$
then $\tau_k\big(\t\lambda\big)=\t{\lambda'}$.
\end{enumerate}
\end{lemma}

\begin{proof}
Let $\mu'=\tau_\l(\mu)$. Then
$$\mu'(t)=\gamma(\mu(\gamma(t)))$$
where $\gamma=\gamma_\l$, i.e., $\gamma:\Z\to\Z,\,x\mapsto \l+1-x$. $\gamma$ is an
order-reversing bijection satisfying 
$\gamma(b')=b+k-1$ and $\gamma(b)=b'+k-1$.

If $\mu$ satisfies (i), then
$\mu'(b')=\gamma(\mu(\gamma(b')))=\gamma(\mu(b+k-1))=\gamma(b)=b'+k-1$
and similarly, $\mu'(b'+k-1)=b'.$ This is condition (i) for
$\mu'.$

Assume that $\mu$ satisfies (ii) for $b$, that is,
$$1\leq i<b\Rightarrow i<\mu(\mu(i)).$$%=\gamma(\mu'_{\gamma(\mu_t)})=\gamma(\mu'_{\mu'_{\gamma(t)}})$$
Therefore
$$\gamma(b)< \gamma(i)\leq\gamma(1)\Rightarrow \gamma(i)>\gamma(\mu(\mu(i))),$$
or equivalently,
$$b'+k-1< \gamma(i)\leq\l\Rightarrow \gamma(i)>\gamma(\mu(\mu(i))).$$
If we set $i'=\gamma(i)$ then $i=\gamma(i')$ and
$$b'+k-1< i'\leq\l\Rightarrow i'>\gamma(\mu(\mu(\gamma(i')))).$$
This is precisely (iii) for $\mu',\,b',\,i'$ since
$$\mu'(\mu'(i'))=\mu'(\mu'(\gamma(i)))=\gamma(\mu(\gamma(\mu'(\gamma(i)))))
=\gamma(\mu(\mu(\gamma(i')))).$$
The other implications in (a) are similar.

The proof of (b) is a straight forward calculation. It suffices to
prove that $\t{\lambda'}(i)+\t\lambda(k+1-i)=k+1,\,1\leq i\leq k$.
In fact:
$$\t{\lambda'}(i)+\t\lambda(k+1-i)=(\tau_\l(\mu)(b'+i-1)-b'+1)+(\mu(b+(k+1-i)-1)-b+1)$$
$$=\gamma(\mu(\gamma(b'+i-1)))+\mu(b+k-i)-b-b'+2$$
$$=\gamma(\mu(b+k-i))+\mu(b+k-i)+k-\l=\l+1-\mu(b+k-i)+\mu(b+k-i)+k-\l=k+1.$$
\end{proof}

\begin{proof}[Proof of Theorem \ref{sq}] First we prove that $\mu\in\P^\l(\lambda)$ implies (i)-(iv).
We proceed by induction on $\l$. If $\l=k$ then  $b=1$, conditions
(ii) and (iii) are vacuum, and $\mu\in\P^\l(\lambda)$ if and only
if $\mu\in\{\lambda,\tau_\l(\lambda)\}$, which is precisely
condition (iv) since $\lambda'=\mu$.

For the inductive step assume that $\l>k$ and $\mu\in\P^\l(\lambda)$, say
$$\mu\in \{(\nu:j)\,|\,1\leq j\leq \nu(\l-1)\}\cup\{\tau_{\l}(\nu:j)\,|\,1\leq j\leq \nu(\l-1)\}$$
where $\nu\in\P^{\l-1}(\lambda)$ satisfies conditions (i)-(iv) for
$b'$, $1\leq b'\leq \l-k$. If $\mu=(\nu:j)$, then take $b=b'$. The
only thing that we have to check is (iii) for $i=\l$ which in this
case reads $\mu(\mu(\l))<\l$, but every part of $\nu$, and
therefore of $\mu$, is less than $\l$.

If $\mu=\tau_{\l}(\nu:j)$ then it follows from Lemma \ref{ii-iii}
and the previous paragraph that $\mu$ satisfies (i)-(iv) for
$b=\l-k-b'+2.$

Now we prove that (i)-(iv) imply that $\mu\in\P^\l(\lambda)$. If
$\l=k$ then we must have $b=1$,(i) means that $\mu\in\Psq^k$, (ii)
and (iii) are vacuum, and (iv) means that $\mu\in\{\lambda,
\tau_\l(\lambda)\}$.

If $\l>k$ then $\mu$ is not a square partition since $\mu(\l)=1$
and $\mu(1)=\l$ imply $\mu(\mu(\l))=\l$ and $\mu(\mu(1))=1$,
forcing $b=1$ and $b+k-1=\l$ and therefore
$\l-1=k-1$, a contradiction. So we consider the cases
$\mu(\l)\neq1$ and $\mu(1)\neq\l$.

If $\mu(1)<\l$ then $\mu=(\nu:\mu(\l))$ where by the induction
hypothesis $\nu\in\P^{\l-1}(\lambda)$ and therefore
$\mu\in\P^{\l}(\lambda).$

If $\mu(1)=\l$ and $\mu(\l)>1$ then $\mu'=\tau_\l(\mu)$ satisfies
$\mu'(\l)=1,\,\mu'(1)<\l$, %so $\nu\setminus\{\nu_\l\}\in\P^{\l-1}(\lambda)$.
and by Lemma \ref{ii-iii}, $\mu'$ satisfies (i)-(iv) for
$b'=\l-k-b+2$. So we can apply the previous paragraph to conclude
that $\mu'\in\P^\l(\lambda)$ and therefore $\mu\in\P^\l(\lambda)$.
\end{proof}

Now we use these results to provide an alternative proof of the following

\begin{lemma}[Theorem \ref{thm1}(iii), \cite{BCDM} Proposition 2.6 (ii)]\label{existence}
Let $\mu\in\P^\l$ for some $\l\in\N$. Then $\mu\in\P^\l(\lambda)$
for some $\lambda\in\Psq^k,\, k\geq1$.
\end{lemma}

\begin{proof}
We are to prove that for every $\mu\in\P^\l$ there exist $b$ and
$k$ satisfying the conditions in Theorem \ref{sq}.
$\mu\circ\mu:[\l]\to[\l]$ is increasing since $\mu$ is decreasing.
It has a fixed point since the sequence $j_1,\,j_2\dots$ defined
by $j_1=1,\,j_{m+1}=\mu(\mu(j_m))$ is increasing and must
stabilize. Let $b$ be the smallest fixed point of $\mu\circ\mu$.
Clearly $\mu(b)$ is another fixed point of $\mu\circ\mu$. We claim
that it is the largest one. Assume that $a>\mu(b)$ is another
fixed point for $\mu\circ\mu$, then so is
$\mu(a)\leq\mu(\mu(b))=b$. The minimality of $b$ implies
$\mu(a)=b$, but this is impossible since $\mu(\mu(a))=a>\mu(b).$

Set $k=\mu(b)-b+1$, so $\mu(b)=b+k-1$ and condition (i) is
satisfied. If $1\leq t< b$ then $\mu(\mu(t))\leq\mu(\mu(b))=b$,
but the inequality cannot occur since $b$ is the smallest fixed
point of $\mu\circ\mu$. This proves condition (ii), and (iii) is
proved similarly. Defining $\lambda\in\Psq^k$ as in condition (iv) we get that $\mu\in\P^\l(\lambda)$
\end{proof}

\subsection{Inequalities for $\boldsymbol{\P^\l(\Omega_m)}$}\

As before, we regard elements of  $\P^\l(\Omega_m)$ as decreasing
functions $\mu:[\l]\to[\l+m-1]$. Here $\tau=\tau_{m+\l-1}$ so
$\tau(\mu)=\gamma_{m+\l-1}\circ\mu\circ\gamma_\l$.

We define $$t:[\l]\times[\l+m-1]\to [\l],\, t(r,s)=\begin{cases}
s&\text{if }s<r\\r&\text{if }r\leq s\leq m+r-1\\s-m+1&\text{if
}s>m+r-1
\end{cases}$$
$t$ has the effect of ``compress" a $\l\times(\l+m-1)$ rectangle
into a $\l\times\l$ square. For example, if $\l=5,\,m=3$, the
values of $t(r,s)$ are shown in Figure \ref{t}, where $r$ and $s$
are displayed vertical and horizontally, respectively.

\begin{figure}[h]
\begin{center}\setlength{\unitlength}{0.5mm}\begin{picture}(70,70)
{\color{gray}
\multiput(0,5)(0,10){6}{\line(1,0){70}}
\multiput(0,5)(10,0){8}{\line(0,1){50}}}
\multiput(3,7)(0,10){5}{1}
\multiput(13,7)(0,10){3}{2}
\multiput(23,7)(0,10){2}{3}
\put(33,7){4}

\multiput(63,7)(0,10){5}{5}
\multiput(53,17)(0,10){4}{4}
\multiput(43,27)(0,10){3}{3}
\multiput(33,37)(0,10){2}{2}

\multiput(3,47)(10,0){3}{1}
\multiput(13,37)(10,0){3}{2}
\multiput(23,27)(10,0){3}{3}
\multiput(33,17)(10,0){3}{4}
\multiput(43,7)(10,0){3}{5}

\put(-5,55){\vector(0,-1){50}} \put(-10,27){$r$}
\put(0,60){\vector(1,0){70}} \put(35,62){$s$}

\thicklines
\put(0,45){\framebox(30,10)}\put(10,35){\framebox(30,10)}
\put(20,25){\framebox(30,10)}\put(30,15){\framebox(30,10)}\put(40,5){\framebox(30,10)}
\end{picture}
\end{center}\caption{$t:[5]\times[7]\to[5]$}\label{t}
\end{figure}

Let $\tilde\mu(i)=t(i,\mu(i))$. In general,
$\tilde\mu:[\l]\to[\l]$ is not a partition since it is not
decreasing.
The following properties are easy to prove

\begin{lemma}\label{propt}
\
\begin{enumerate}
\item[(a)] $t(r,s)<r$ if and only if $s<r$, and $t(r,s)>r$ if and only
if $s>r+m-1$. \item[(b)]
$t\circ(\gamma_\l\times\gamma_{\l+m-1})=\gamma_\l\circ\mu.$
\item[(c)] Let $\mu'=\tau(\mu)$. Then
$\gamma_\l\circ\widetilde{\mu'}=\t\mu\circ\gamma_{\l}$.
\end{enumerate}
\end{lemma}

\begin{proof}
(a) is clear from the definition of $t$. For (b), notice that
\begin{multline*}t(\gamma_\l(r),\gamma_{\l+m-1}(s))\\=\begin{cases}
\l+m-s&\text{if }\l+m-s<\l+1-r\\\l+1-r&\text{if }\l+1-r\leq
\l+m-s\leq m+(\l+1-r)-1\\(\l+m-s)-m+1&\text{if
}\l+m-s>m+(\l+1-r)-1
\end{cases}\end{multline*}
$$=\begin{cases}
\gamma_\l(s-m+1)&\text{if }r+m-1<s\\
\gamma_\l(r)&\text{if }r\leq s\leq m+r-1=\gamma_\l(\mu(r,s))\\
\gamma_\l(s)&\text{if }s<r
\end{cases}$$
The conclusion follows.

(c) follows from Lemma (b) and the definition of $\mu'$:
$$\gamma_\l\big(\widetilde{\mu'}(i)\big)=\gamma_\l(t(i,\mu'(i)))=t(\gamma_\l(i),\gamma_{\l+m-1}(\mu'(i)))
=t(\gamma_\l(i),\mu(\gamma_{\l}(i)))=\t\mu(\gamma_{\l}(i))$$

\end{proof}

Now we provide a characterization for
$\P^\l(\Omega_m).$

\begin{thm}\label{tilde}
Let $\mu\in\P^{\l,\l+m-1}.$ Then $\mu\in\P^\l(\Omega_m)$ if and
only if $\tilde\mu$ satisfies the following conditions for every
$i\in[\l]$:
\begin{enumerate}
\item[(i)] If $\tilde\mu(i)>i$ then $\tilde\mu(\tilde\mu(i))> i$.
\item[(ii)]  If $\tilde\mu(i)<i$ then $\tilde\mu(\tilde\mu(i))<
i$.
\end{enumerate}
\end{thm}

First we prove the following:
\begin{claim}\label{tau}
$\mu\in\P^{\l,m+\l-1}$ satisfies (ii) if and only if $\tau(\mu)$
satisfies (i).
\end{claim}

\begin{proof}
Assume that $\mu$ satisfies (ii) and that $\widetilde{\mu'}(i)>i$
where $\mu'=\tau(\mu)$. From Lemma \ref{propt} (c) and the fact
that $\gamma_\l$ is order-reversing we see that
$\gamma_\l(i)>\gamma_\l\big(\widetilde{\mu'}(i)\big)=\t{\mu}(\gamma_\l(i))$
and as a consequence of (i),
$\gamma_\l(i)>\t{\mu}(\t{\mu}(\gamma_\l(i)))$. Applying
$\gamma_\l$ again we get
$i<\gamma_\l(\t{\mu}(\t{\mu}(\gamma_\l(i))))$, but
$\gamma_\l(\t{\mu}(\t{\mu}(\gamma_\l(i))))=\widetilde{\mu'}(\gamma_\l(\t{\mu}(\gamma_\l(i))))
=\widetilde{\mu'}\big(\widetilde{\mu'}(i)\big).$ The conclusion
follows.
\end{proof}

\begin{proof}[Proof of Lemma \ref{tilde}]
First we prove by induction that every $\mu\in\P^\l(\Omega_m)$ satisfies
(i) and (ii). If $\l=1$ then (i) and (ii) are vacuum since
$\t\mu(1)=1.$

Because of Claim \ref{tau} we just have to consider
$\mu=(\nu:j),\,1\leq j\leq\nu(\l-1)$ to complete the induction
step. Clearly $\t\mu(i)=\t\nu(i)\in[\l-1]$ for $i\in[\l-1]$, and
the premise of (i) is impossible if $i=\l$, so we just have to
prove that (ii) holds for $i=\l$. But $\tilde\mu(\l)<\l$ implies
$\tilde\mu(\tilde\mu(\l))=\tilde\nu(\tilde\mu(\l))\leq\l-1< \l$
since $\t\nu:[\l-1]\to[\l-1].$

Now we prove that if $\mu\in\P^{\l,m+\l-1}$ is so that $\tilde\mu$
satisfies (i) and (ii), then $\mu\in\P^\l(\Omega_m)$. For $\l=1$
there is nothing to prove since $\P^1(\Omega_m)=\P^{1,m}$. Assume
$\l>1$. It is not possible to have $\t\mu(1)=\l$ and $\t\mu(\l)=1$
since this would contradict (i) and (ii). If $\t\mu(1)=\l$ then if
follows from Lemma \ref{propt} (c) and Claim \ref{tau} that we can
replace $\mu$ by $\tau(\mu)$. Therefore we can assume that
$\t\mu(1)<\l$. From the definition of $t$ if follows that
$\mu(1)<\l+m-1$ and therefore no part of $\mu$ is bigger than
$\l+m-1$. So $\nu=\mu\setminus\mu_\l\in \P^{\l-1,m+\l-2}$ and it
satisfies (i) and (ii). By the induction hypothesis
$\nu\in\P^{\l-1}(\Omega_m)$ and therefore $\mu=(\nu:\mu_j)\in
\P^\l(\Omega_m)$.
\end{proof}

\begin{rem}\label{remlambda1}
Clearly $\P^\l(\Omega_1)=\P^\l(\lambda)$ where $\lambda=(1)$ is
the only element in $\Psq^1$. In this case Lemma \ref{sq} states
that $\mu\in\P^\l(1)$  if and only if:
\begin{enumerate}
\item[(i) and (iv)] There exists a positive integer $b\in[\l]$ so that $\mu(b)=b$, \item[(ii)]
 If $1\leq i< b$ then $\mu(\mu(i))>i$, \item[(iii)] If
$b< i\leq \l$ then $\mu(\mu(i))<i$ .
\end{enumerate}

Since $\t\mu=\mu$ for $m=1$, Lemma \ref{tilde} states that
$\mu\in\P^\l(\Omega_1)$ if and only if $\mu$ satisfies the
following conditions for every $i\in[\l]$:
\begin{enumerate}
\item[(i)] If $\mu(i)>i$ then $\mu(\mu(i))> i$. \item[(ii)]  If
$\mu(i)<i$ then $\mu(\mu(i))< i$.

\end{enumerate}

We claim that both sets of conditions coincide:

``Conditions in Lemma \ref{sq} $\Rightarrow$ Conditions in Lemma
\ref{tilde}": If $i\leq b$ then $\mu(i)\geq\mu(b)=b\geq i$.
Therefore $\mu(i)<i$ implies $i>b$ which in turn implies
$\mu(\mu(i))< i$ and similarly $\mu(i)>i$ implies $\mu(\mu(i))>i$.

``Conditions in Lemma \ref{tilde} $\Rightarrow$ Conditions in
Lemma \ref{sq}": Lemma \ref{existence} implies that
$\mu\in\P^\l(\lambda)$ for some $\lambda\in\Psq^k$, and therefore
$(\mu\circ\mu)(b)=b$ for some $b\in[\l]$, but the conditions in
Lemma \ref{tilde} imply that $\mu(b)=b$ and therefore $k=1.$ Since
$\mu$ is increasing, $i< b$ (resp.\ $i>b$) implies that
$\mu(i)\geq\mu(b)=b>i$ (resp.\ $\mu(i)\leq\mu(b)=b<i$). and
therefore $\mu(\mu(i))>i$ (resp.\ $\mu(\mu(i))<i$).
\end{rem}

\section{Trees, forests and Catalan numbers} \label{bijections}

In this section we will define a bijection between
$\P^\l(\lambda)$ (resp.\ $\P^\l(\Omega_m)$) and a family
enumerated by the Catalan (resp.\ ballot) numbers.
The Catalan numbers appear in various counting problems, see
\cite{Stanley} for a 66 interpretations of the Catalan numbers. Some of these generalize to the ballot numbers.

Consider the set
$\T_\l$ of pairs $(T^-,T^+)$ of rooted trees with a total of
$\l-1$ edges. These pairs are in bijection with trees with $\l$
edges: cutting the rightmost branch from the root creates such a
pair, and viceversa, we can attach a tree as the rightmost subtree
of the root (See Figure \ref{2trees} for an example). Therefore
there are $c_{\l}$ pairs of rooted trees with a total of $\l-1$
edges. We will use $\T_\l$ to establish a bijection in order to prove Theorem \ref{thm1}(i).

\begin{figure}[h]
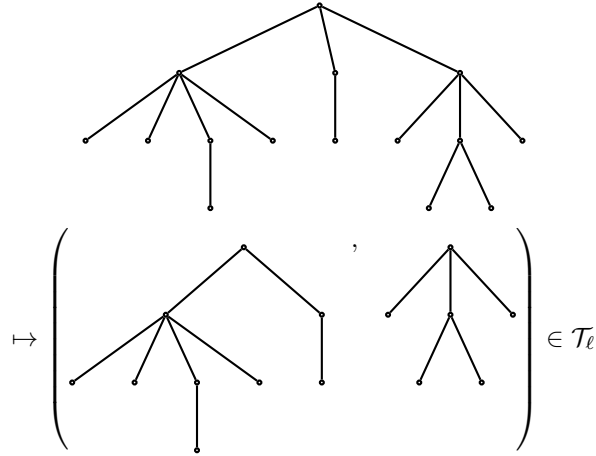

\begin{center}
\psset{radius=.4mm} \pstree[levelsep=0.9cm]{\TC}{
\pstree{\TC}{\TC\TC\pstree{\TC}{\TC}\TC} \pstree{\TC}{\TC}
\pstree{\TC}{\TC\pstree{\TC}{\TC\TC}\TC} }
$$\mapsto\begin{pmatrix}
\psset{radius=0.4mm} \pstree[levelsep=0.9cm]{\TC}{
\pstree{\TC}{\TC\TC\pstree{\TC}{\TC}\TC} \pstree{\TC}{\TC}
 }&
,
 & \pstree[levelsep=0.9cm]{\TC}{\TC\pstree{\TC}{\TC\TC}\TC}
\end{pmatrix}\in\T_\l$$
%\Bigg(\psset{radius=0.4mm} \pstree[levelsep=0.9cm]{\TC}{
%\pstree{\TC}{\TC\TC\pstree{\TC}{\TC}\TC} \pstree{\TC}{\TC}
% }, \pstree[levelsep=0.9cm]{\TC}{\TC\pstree{\TC}{\TC\TC}\TC} \Bigg)
\end{center}\caption{Cutting a tree in two.}\label{2trees}
\end{figure}

%One of the main results of \cite{BCDM} is the following.
%\begin{thm}[\cite{BCDM}]
%$$\#\P^\l(\lambda)=\begin{cases}
%c_{\l-k+1} & \lambda=\tau(\lambda)\\
%2c_{\l-k+1} & \lambda\neq\tau(\lambda)
%\end{cases},\quad \#\P^\l(\Omega_m)=b_{\l,m-1}$$
%where
%$$b_{\l,m}={{2\l+m}\choose{\l}}-{{2\l+m}\choose{\l-1}}\quad\text{and}\quad
%c_\l=b_{\l,0}$$ are the ballot and Catalan numbers, respectively.
%\end{thm}

The generalized Catalan numbers are defined by the formula
$C_{k,\gamma}(n)=\frac{\gamma}{nk+\gamma} {{kn+\gamma}\choose{n}}$
(see \cite{GS2006, Chen2008, Stanley}). $C_{k,\gamma}(n)$ is the number of ordered forests
with
 $\gamma$ $k$-ary trees and with total
number of $n$ internal vertices. The ballot numbers are a special
case of the generalized Catalan numbers since $b_{\l,m-1}=
C_{2,m}(\l)$. Therefore $b_{\l,m-1}$ is the number of ordered
forests with
 $m$ binary trees and with total number of $\l$ internal vertices, or equivalently,
 the number of ordered forests with $m$ trees and with total number of $\l$ edges, since there is a
bijection between binary trees with $n$ vertices and rooted trees
with $n$ edges. As before, we can cut non-empty trees to obtain
pairs of trees. We will use this to establish Theorem \ref{thm2}.

\subsection{$\boldsymbol{\P^\l(\lambda)}$ and Trees}
\begin{rem} Let $\lambda\in\Psq^k$.
It follows from Theorem \ref{sq} that we can define a map
$\theta:\P^{\l}(\lambda)\to\P^{\l-k+1}(1)$ by
$$\theta(\mu)(i)=\begin{cases} \mu(i)-k+1&1\leq i\leq
b\\\mu(i+k-1)&b\leq i\leq\l-k+1
\end{cases}$$

The effect of $\theta$ is to ``remove"  a $k\times k$ square from
$\mu$ and replace it with a $1\times 1$ square which becomes a
fixed point of the partition.
\end{rem}

\begin{example}
Let $\mu=(7 , 6 , 5 , 3 , 3 , 3 , 3 , 3 , 3 , 1)$.
$b=3,\,\mu(b)=5=b+k-1$, therefore $k=3$,
$\mu\in\P^3(\lambda)=\P^3(\tau(\lambda))$ where $\lambda=(3,1,1)$,
and $\theta(\mu)=(5,4,3,3,3,3,3,1)$. See Figure
\ref{remove}.\end{example}

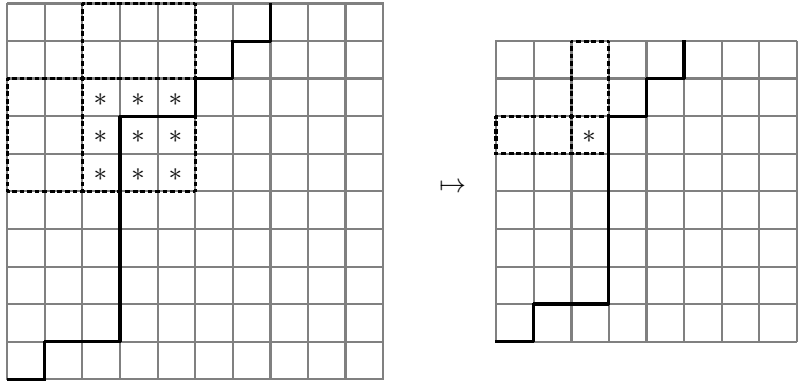
\begin{figure}[h]
\begin{center}
\setlength{\unitlength}{0.5mm}\begin{picture}(210,110)
{\color{gray} \multiput(0,0)(0,10){11}{\line(1,0){100}}
\multiput(0,0)(10,0){11}{\line(0,1){100}}}

\thicklines

\put(00,50){\dashbox(50,30)}
\put(20,50){\dashbox(30,50)}

\put(0,0){\line(1,0){10}} \put(10,10){\line(1,0){20}}
\put(30,70){\line(1,0){20}}
\put(50,80){\line(1,0){10}}\put(60,90){\line(1,0){10}}

\put(10,10){\line(0,-1){10}} \put(30,70){\line(0,-1){60}}
\put(50,80){\line(0,-1){10}}\put(60,90){\line(0,-1){10}}
\put(70,100){\line(0,-1){10}}

\put(115,50){$\mapsto$} \thinlines

{\color{gray} \multiput(130,10)(0,10){9}{\line(1,0){80}}
\multiput(130,10)(10,0){9}{\line(0,1){80}}}

\thicklines

\put(130,10){\line(1,0){10}} \put(140,20){\line(1,0){20}}
\put(160,70){\line(1,0){10}} \put(170,80){\line(1,0){10}}

\put(140,20){\line(0,-1){10}} \put(30,70){\line(0,-1){60}}
\put(160,70){\line(0,-1){50}}\put(170,80){\line(0,-1){10}}
\put(180,90){\line(0,-1){10}}

\multiput(23,51)(10,0){3}{*} \multiput(23,61)(10,0){3}{*}
\multiput(23,71)(10,0){3}{*}

\put(153,61){*}
\put(130,60){\dashbox(30,10)}
\put(150,60){\dashbox(10,30)}

\end{picture}
\end{center}\caption{An example of the action of $\theta$.}\label{remove}
\end{figure}

Conditions (i)-(iii) in Theorem \ref{sq} guarantee that $\theta$
is well defined and surjective. (iv) implies that the fibers of
$\theta$ have one or two elements depending of whether or not
$\lambda=\tau(\lambda)$. This proves that
$$\#\P^\l(\lambda)=\begin{cases}
\#\P^{\l-k+1}(1) & \lambda=\tau(\lambda)\\
2\#\P^{\l-k+1}(1) & \lambda\neq\tau(\lambda)
\end{cases}.$$
To establish Theorem \ref{thm1}(iii), we just need to prove
that $\#\P^{\l}(1)=c_{\l}$. In order to do this, we will define a bijection between $\P^{\l}(1)$ and $\T_\l$.

Now we describe how to create an element of $\P^\l(1)$ from a pair
of rooted trees $(T^-,T^+)$. Number the levels of both trees so
that the roots are located in level 1. Starting from the deepest
level and moving up and from left to right, label the vertices of
the odd levels of $T^-$ and the even levels of $T^+$ consecutively
with the numbers $1,2,3\dots$, and label the vertices of the even
levels of $T^-$ and the odd levels of $T^+$ consecutively with the
numbers $\l, \l-1,l-2\dots.$ Then label the roots of $T^\pm$ with
$b^\pm$, where $b\in[\l]$ is the only number that has not been
used so far. To define $\mu:[\l]\to[\l]$ let $\mu(i)$ be the
parent of $i$ (if the patent of $i$ is $b^\pm$, define
$\mu(i)=b$), and $\mu(b)=b$.

\begin{example}\label{ex14} Consider the pair of trees in Figure
\ref{2trees}. After numbering the nodes, we get the pair of trees
in Figure \ref{2labelledtrees}.

\begin{figure}[h]
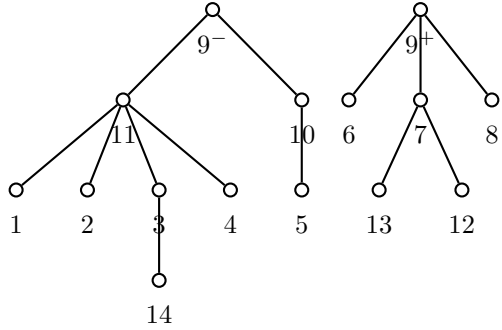

\begin{center}
\psset{radius=1mm} \pstree[levelsep=1.2cm]{\TC~{$9^-$}}{
\pstree{\TC~{11}}{\TC~{1}\TC~{2}\pstree{\TC~{3}}{\TC~{14}}\TC~{4}}
\pstree{\TC~{10}}{\TC~{5}}
 }\quad
 \pstree[levelsep=1.2cm]{\TC~{$9^+$}}{\TC~{6}\pstree{\TC~{7}}{\TC~{13}\TC~{12}}\TC~{8}}
\end{center}\caption{A pair of labelled trees}\label{2labelledtrees}
\end{figure}

The corresponding partition is
%Consider the following partition $\mu\in\P^{14}(1)$
\begin{center}
\begin{tabular}{|c||c|c|c|c|c|c|c|c|c|c|c|c|c|c|c|c|c|c|c|c|c|c|c|c|c|}\hline
$i$&1&2&3&4&5&6&7&8&9&10&11&12&13&14\\\hline
$\mu(i)$&11&11&11&11&10&9&9&9&9&9&9&7&7&3\\\hline
\end{tabular}
\end{center}
%Here $L=\{1,\dots,5\},\,b=6,\,H=\{7,\dots,14\}$. The corresponding pair of trees is
\end{example}

\begin{lemma}
The map $\mu:[\l]\to[\l]$ obtained by this process is a partition in $\P^\l(1)$.
\end{lemma}

\begin{proof}
To simplify the notation, let $L=[b-1],\,H=[\l]\setminus[b].$ We
have to prove that $\mu$ is decreasing and satisfies the
conditions in Remark \ref{remlambda1}.

It is clear form the construction that $\mu(b)=b$, $i\in L$
implies $\mu(i)\notin L$, and $i\in H$ implies $\mu(i)\notin H$.
Let $1\leq i\leq j<\l$. If $i\in L$ and $j\in H$ then
$\mu(i)\in\{b\}\cup H$ and $\mu(j)\in\{b\}\cup L$. Therefore
$\mu(i)\geq\mu(j).$ Similar considerations hold if one among $i,j$
is equal to $b$. Assume that $i,j\in L$. Then either $i$ is
located at a level deeper than $j$, or they are in the same level
but $i$ is to the left of $j$, and the same is true about their
parents. But their parents $\mu(i),\mu(j)\in\{b\}\cup H$, and
therefore $\mu(i)\geq\mu(j)$. A similar argument works if we
assume $i,j\in H$.

The conditions in Remark \ref{remlambda1} say
\begin{enumerate}
\item[(i)] $\mu(b)=b$
\item[(ii)] $i\in L\Rightarrow\mu(\mu(i))>i.$
\item[(iii)] $i\in H\Rightarrow\mu(\mu(i))<i.$
\end{enumerate}

(i) follows from the construction of $\mu$. If $i\in L$ then
either $\mu(\mu(i))$ is two levels above $i$ and therefore
$\mu(\mu(i))>i$, or $\mu(i)=b$, so $\mu(\mu(i))=b>i$. Similar
considerations work if $i\in H.$
\end{proof}

The process can be reversed: for $\mu\in\P^\l(1)$ let $b$ be its
fixed point, and $L,\,H$ as before. We define a pair of trees
$(T^-,T^+)$ with $i\in([l]\setminus\{b\})\cup\{b^-,b^+\}$
as vertex set. The root of $T^\pm$ is $b^\pm$, and the edges are
drawn according to the following rules:
\begin{enumerate}
\item  If $i\in[l]\setminus\{b\}$ and $\mu(i)\neq b$, we draw an
edge $i\to\mu(i)$. \item If $i\in L$ and $\mu(i)=b$, we draw an
edge $i\to b^+$. \item If $i\in H$ and $\mu(i)=b$, we draw an edge
$i\to b^-$.
\end{enumerate}
The vertices in $L$ (resp.\ $H$) are organized increasingly
(resp.\ decreasingly) from left to right. This procedure creates a
bijection between $\T_\l$ and  $\P^\l(1)$

\begin{example}
The following are the $14=C_4$ pairs of rooted trees with 3 edges.

\bigskip
%$\Bigg(\pstree[levelsep=0.5cm]{\Tcircle{}}\ \  ,\,
%\pstree[levelsep=0.5cm]{\Tcircle{} }{ \pstree{ \Tcircle{} }{
%\Tcircle{} \Tcircle{}}}\Bigg)$
%\begin{figure}[h]\label{c4}
 \psset{radius=1mm}
 \pstree[levelsep=0.5cm]{\Tcircle{}} \
{\pstree[levelsep=0.5cm]{\Tcircle{} }{ \pstree{ \Tcircle{} }{
\Tcircle{} \Tcircle{}}} }
 \qquad
 \pstree[levelsep=0.5cm]{\Tcircle{}
}\ {\pstree[levelsep=0.5cm]{\Tcircle{} }{ \pstree{ \Tcircle{} }{
\Tcircle{}}\Tcircle{}}}
 \qquad \pstree[levelsep=0.5cm]{\Tcircle{}
}\ {\pstree[levelsep=0.5cm]{\Tcircle{} }{
         \Tcircle{}
        \pstree{ \Tcircle{} }{
                     \Tcircle{}}}}
\qquad \pstree[levelsep=0.5cm]{\Tcircle{} } \ \
{\pstree[levelsep=0.5cm]{\Tcircle{} }{
        \pstree{ \Tcircle{} }{
        \pstree{ \Tcircle{} }{
                     \Tcircle{}
                              }
                    }}}
\qquad \pstree[levelsep=0.5cm]{\Tcircle{} }\
{\pstree[levelsep=0.5cm]{\Tcircle{} }{\Tcircle{}
        \Tcircle{}\Tcircle{}}}
 \newline \bigskip

{\pstree[levelsep=0.5cm]{\Tcircle{} }{
        \pstree{ \Tcircle{} }{
        \pstree{ \Tcircle{} }{
                     \Tcircle{}
                              }
                    }}}\ \pstree[levelsep=0.5cm]{\Tcircle{} } \qquad\qquad
 \pstree[levelsep=0.5cm]{\Tcircle{} }{
        \pstree{ \Tcircle{} }{
                     \Tcircle{}\Tcircle{}
                              }}\
        \pstree[levelsep=0.5cm]{\Tcircle{}}
 \qquad\qquad \pstree[levelsep=0.5cm]{\Tcircle{} }{
        \pstree{ \Tcircle{} }{
                     \Tcircle{}
                              }}\
        \pstree[levelsep=0.5cm]{ \Tcircle{} }{\Tcircle{}}
 \qquad \pstree[levelsep=0.5cm]{ \Tcircle{} }{\Tcircle{}}\
 \pstree[levelsep=0.5cm]{\Tcircle{} }{
        \pstree{ \Tcircle{} }{
                     \Tcircle{}
                              }}
 \qquad \pstree[levelsep=0.5cm]{\Tcircle{} }{
         \Tcircle{}}\
        \pstree[levelsep=0.5cm]{ \Tcircle{} }{
                     \Tcircle{}\Tcircle{}
                              }
\qquad \pstree[levelsep=0.5cm]{\Tcircle{} }{\pstree{ \Tcircle{}
}{
                     \Tcircle{}
                              }
        \Tcircle{}}\ \pstree[levelsep=0.5cm]{\Tcircle{}}

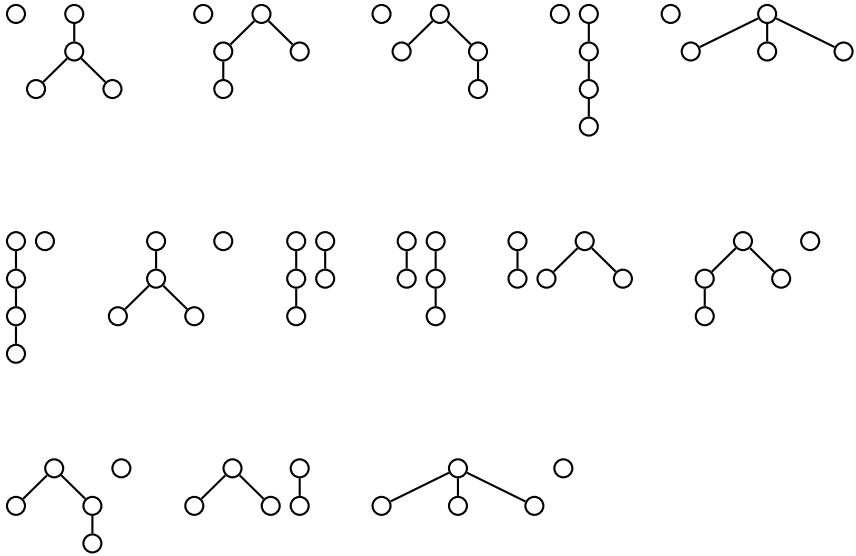
\newline\bigskip

 {\pstree[levelsep=0.5cm]{\Tcircle{} }{
         \Tcircle{}
        \pstree{ \Tcircle{} }{
                     \Tcircle{}}}}\ \pstree[levelsep=0.5cm]{\Tcircle{} }
\quad\qquad \pstree[levelsep=0.5cm]{ \Tcircle{} }{
                     \Tcircle{}\Tcircle{}
                              }\  \pstree[levelsep=0.5cm]{\Tcircle{} }{
         \Tcircle{}}
\qquad {\pstree[levelsep=0.5cm]{\Tcircle{} }{\Tcircle{}
        \Tcircle{}\Tcircle{}}}\ \pstree[levelsep=0.5cm]{\Tcircle{} }

\bigskip
The following are the labeling of the nodes following the
algorithm described before, and the corresponding partitions.
These are in fact the 14 partitions in $\P^4(1).$
\bigskip

%\begin{figure}[h]\label{c4part}
 \pstree[levelsep=1cm]{\Tcircle{$2^-$}} \
{\pstree[levelsep=1cm]{\Tcircle{$2^+$} }{ \pstree{ \Tcircle{1} }{
\Tcircle{4} \Tcircle{3}}} }
 $\mu=(2,2,1,1)$\qquad\qquad
 \pstree[levelsep=1cm]{\Tcircle{$3^-$}
}\ {\pstree[levelsep=1cm]{\Tcircle{$3^+$} }{ \pstree{ \Tcircle{1} }{
\Tcircle{4}}\Tcircle{2}}}$\mu=(3,3,3,1)$
 \\ \pstree[levelsep=1cm]{\Tcircle{$3^-$}
}\ {\pstree[levelsep=1cm]{\Tcircle{$3^+$} }{
         \Tcircle{1}
        \pstree{ \Tcircle{2} }{
                     \Tcircle{4}}}}$\mu=(3,3,3,2)$
\qquad\qquad \pstree[levelsep=1cm]{\Tcircle{$3^-$} }\
{\pstree[levelsep=1cm]{\Tcircle{$3^+$} }{
        \pstree{ \Tcircle{2} }{
        \pstree{ \Tcircle{4} }{
                     \Tcircle{1}
                              }
                    }}}\ $\mu=(4,3,3,2)$\\
\qquad  \pstree[levelsep=1cm]{\Tcircle{$4^-$} }\
{\pstree[levelsep=1cm]{\Tcircle{$4^+$} }{\Tcircle{1}
        \Tcircle{2}\Tcircle{3}}}\ $\mu=(4,4,4,4)$
\qquad\qquad
{\pstree[levelsep=1cm]{\Tcircle{$2^-$} }{
        \pstree{ \Tcircle{3} }{
        \pstree{ \Tcircle{1} }{
                     \Tcircle{4}
                              }
                    }}}\ \pstree[levelsep=1cm]{\Tcircle{$2^+$} }\ \quad
                    $\mu=(3,2,2,1)$\
\\\\
 \pstree[levelsep=1cm]{\Tcircle{$3^-$} }{
        \pstree{ \Tcircle{4} }{
                     \Tcircle{1}\Tcircle{2}
                              }}\
        \pstree[levelsep=1cm]{\Tcircle{$3^+$}}\ \quad $\mu=(4,4,3,3)$
 \qquad\qquad \pstree[levelsep=1cm]{\Tcircle{$3^-$} }{
        \pstree{ \Tcircle{4} }{
                     \Tcircle{1}
                              }}\
        \pstree[levelsep=1cm]{ \Tcircle{$3^+$} }{\Tcircle{2}}\ \quad $\mu=(4,3,3,3)$
\\\\
 \qquad \pstree[levelsep=1cm]{ \Tcircle{$2^-$} }{\Tcircle{3}}\
 \pstree[levelsep=1cm]{\Tcircle{$2^+$} }{
        \pstree{ \Tcircle{1} }{
                     \Tcircle{4}
                              }}\ \quad $\mu=(2,2,2,1)$
 \qquad\qquad \pstree[levelsep=1cm]{\Tcircle{$3^-$} }{
         \Tcircle{4}}\
        \pstree[levelsep=1cm]{ \Tcircle{$3^+$} }{
                     \Tcircle{1}\Tcircle{2}
                              }\ \quad $\mu=(3,3,3,3)$\
\\\\
\qquad \pstree[levelsep=1cm]{\Tcircle{$2^-$} }{\pstree{
\Tcircle{4} }{
                     \Tcircle{1}
                              }
        \Tcircle{3}}\ \pstree[levelsep=1cm]{\Tcircle{$2^+$}} \ \quad $\mu=(4,2,2,2)$\ \qquad\qquad
 {\pstree[levelsep=1cm]{\Tcircle{$2^-$} }{
         \Tcircle{4}
        \pstree{ \Tcircle{3} }{
                     \Tcircle{1}}}}\ \pstree[levelsep=1cm]{\Tcircle{$2^+$} }\ \quad $\mu=(3,2,2,2)$ %!!!!!
                    \qquad\qquad
 \pstree[levelsep=1cm]{ \Tcircle{$2^-$} }{
                     \Tcircle{4}\Tcircle{3}
                              }\  \pstree[levelsep=1cm]{\Tcircle{$2^+$} }{
         \Tcircle{1}}\ \quad $\mu=(2,2,2,2)$
\qquad {\pstree[levelsep=1cm]{\Tcircle{$1^-$} }{\Tcircle{4}
        \Tcircle{3}\Tcircle{2}}}\ \pstree[levelsep=1cm]{\Tcircle{$1^+$} } \ \quad $\mu=(1,1,1,1)$
%\caption{The 14 partitions in $\P^4(1)$.}
%\end{figure}
\end{example}

\bigskip
Now we want to describe the action of $\tau:\P^\l(1)\to\P^\l(1)$.
$\tau$ replaces the label $i$ by $\l+1-i$, except $b^-$ and $b^+$
which are replaced by $(\l+1-b)^+$ and $(\l+1-b)^-$, respectively.
The two trees exchange positions and the levels that contained the
increasing sequence $1,2,3,\dots$ now contain the decreasing
sequence $\l,\l-1,\l-2,\dots$ and viceversa.

\begin{example} Consider the partition in Example \ref{ex14}:
$$\mu=(11,11,11,11,10,9,9,9,9,9,9,7,7,3)$$
$$\tau(\mu)=(12,8,8,6,6,6,6,6,6,5,4,4,4,4)$$
Their corresponding pairs of trees are shown in Figure
\ref{extau}.
\end{example}
\begin{figure}[h]
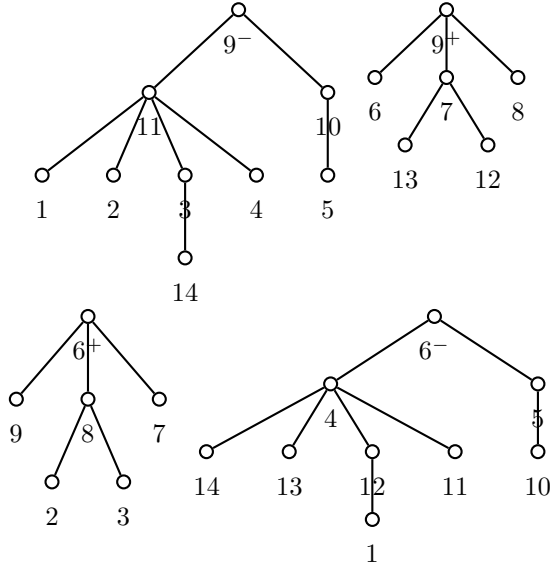

\begin{center}
\psset{radius=1mm} \pstree[levelsep=1.1cm]{\TC~{$9^-$}}{
\pstree{\TC~{11}}{\TC~{1}\TC~{2}\pstree{\TC~{3}}{\TC~{14}}\TC~{4}}
\pstree{\TC~{10}}{\TC~{5}}
 }\quad
 \pstree[levelsep=0.9cm]{\TC~{$9^+$}}{\TC~{6}\pstree{\TC~{7}}{\TC~{13}\TC~{12}}\TC~{8}}
\end{center}
\begin{center}\psset{radius=1mm}\pstree[levelsep=1.1cm]{\TC~{$6^+$}}{\TC~{9}
\pstree{\TC~{8}}{\TC~{2}\TC~{3}}\TC~{7}}\quad
 \pstree[levelsep=0.9cm]{\TC~{$6^-$}}{
\pstree{\TC~{4}}{\TC~{14}\TC~{13}\pstree{\TC~{12}}{\TC~{1}}\TC~{11}}
\pstree{\TC~{5}}{\TC~{10}}
 }
\end{center}
\caption{The action of $\tau$}\label{extau}
\end{figure}

\subsection{$\boldsymbol{\P^\l(\Omega_m)}$ and forests}

Now we define a bijection between $\P^\l(\Omega_m)$ and a family
enumerated by ballot numbers.

Now we describe how to associate to $\mu\in\P^\l(\Omega_m)$ a
$m$-tuple of (possibly empty) pairs of rooted trees.  For
$\mu\in\P^\l(\Omega_m)$, let
$L=\{i\,|\,\tmu(i)>i\},\,M=\{i\,|\,\tmu(i)=i\},\,H=\{i\,|\,\tmu(i)<i\}$.
We claim that $L,\,M$ and $H$ are intervals. Assume that $i\in H$
and $j>i$. From Lemma \ref{propt} (a) we have that
$i>\tmu(i)=t(i,\mu(i))$ implies $i>\mu(i)$ and therefore
$j>i>\mu(i)\geq\mu(j)$ and therefore $\tmu(j)=\mu(j)$ and $j\in
H$. Similarly, $i\in L$ and $j<i$ imply $j\in L.$ This proves that
$L,\,M$ and $H$ are intervals.

The roots of the pairs of trees are going to be $\{b^-\,|\,b\in
M\}\cup \{b^+\,|\,b\in M\}$. The vertices that are not roots are
going to be $L\cup H.$ The edges are drawn according to the
following rules:
\begin{enumerate}
\item  If $i\in L\cup H$ and $\mu(i)\notin M$, we draw an edge $i\to\tmu(i)$.
\item If $i\in L$ and $\tmu(i)=b\in M$, we draw an edge $i\to b^+$.
\item If $i\in H$ and $\tmu(i)=b\in M$, we draw an edge $i\to b^-$.
\end{enumerate}
Now we pair the trees with roots $b^-,\,b^+\,(b\in M)$, and we say
that this pair is the $(m-\mu(b)+b)$-th one. Clearly $b\in M$
implies $m-\mu(b)+b\in[m]$ by the definition of $t$, and if $b<b'$
then $m-\mu(b)+b<m-\mu(b')+b'$, so different pairs have different
positions between 1 and $m$.

\begin{example}  Consider the following partition $\mu\in\P^{24}(\Omega_4)$:
\begin{center}
\begin{tabular}{|c||c|c|c|c|c|c|c|c|c|c|c|c|c|c|c|c|c|c|c|c|c|c|c|}\hline
$i$&1&2&3&4&5&6&7&8&9&10&11&12\\\hline
$\mu(i)$&22&22&22&21&21&21&21&20&17&17&17&16\\\hline
\end{tabular}
\begin{tabular}{|c||c|c|c|c|c|c|c|c|c|c|c|c|c|c|c|c|c|c|c|c|c|c|c|}\hline
$i$&13&14&15&16&17&18&19&20&21&22&23&24\\\hline
$\mu(i)$&16&15&15&15&14&14&13&12&6&3&3&3\\\hline
\end{tabular}
\end{center}
Here, $L=\{1,\dots,12\},\,M=\{13,14,15\},\,H=\{16,\dots,24\}$ and $\tmu:[24]\to[24]$ is
\begin{center}
\begin{tabular}{|c||c|c|c|c|c|c|c|c|c|c|c|c|c|c|c|c|c|c|c|c|c|c|c|}\hline
$i$&1&2&3&4&5&6&7&8&9&10&11&12\\\hline
$\tmu(i)$&19&19&19&18&18&18&18&17&14&14&14&13\\\hline
\end{tabular}
\begin{tabular}{|c||c|c|c|c|c|c|c|c|c|c|c|c|c|c|c|c|c|c|c|c|c|c|c|}\hline
$i$&13&14&15&16&17&18&19&20&21&22&23&24\\\hline
$\tmu(i)$&13&14&15&15&14&14&13&12&6&3&3&3\\\hline
\end{tabular}
\end{center}

This map determines 6 trees:

\bigskip

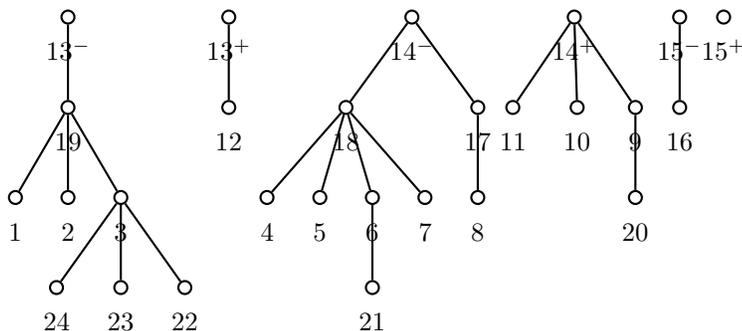
\psset{radius=1mm} \pstree[levelsep=1.2cm,
treesep=0.5cm]{\TC~{$13^-$}}{
\pstree{\TC~{19}}{\TC~{1}\TC~{2}\pstree{\TC~{3}}{\TC~{24}\TC~{23}\TC~{22}}}}
\pstree[levelsep=1.2cm,treesep=0.5cm]{\TC~{$13^+$}}{\TC~{12}}
\pstree[levelsep=1.2cm, treesep=0.5cm]{\TC~{$14^-$}}{
\pstree{\TC~{18}}{\TC~{4}\TC~{5}\pstree{\TC~{6}}{\TC~{21}}\TC~{7}}
\pstree{\TC~{17}}{\TC~{8}} } \pstree[levelsep=1.2cm,
treesep=0.5cm]{\TC~{$14^+$}}{\TC~{11}\TC~{10}\pstree{\TC~{9}}{\TC~{20}}}
\pstree[levelsep=1.2cm, treesep=0.5cm]{\TC~{$15^-$}}{\TC~{16}}
\TC~{$15^+$}

The pairs with  $13^\pm,\, 14^\pm$ and $15^\pm$ as roots are the
first, third and fourth, respectively, while the second pair is
empty. These 4 pairs of trees correspond to a forest with 4 trees
and total number of 24 edges:

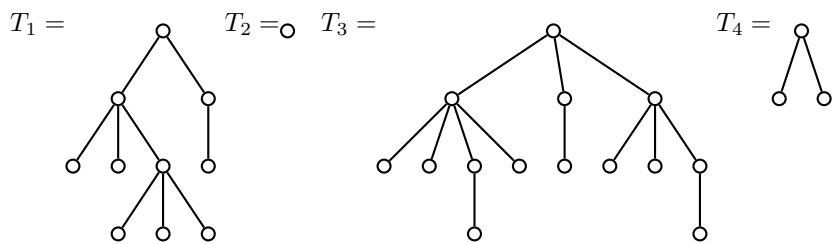
\begin{center}
\psset{radius=1mm} $T_1=$\pstree[levelsep=0.9cm, treesep=0.4cm]{\TC}{
\pstree{\TC}{\TC\TC\pstree{\TC}{\TC\TC\TC}}
\pstree{\TC}{\TC}
 } $T_2=$\pstree[levelsep=0.9cm, treesep=0.4cm]\TC \ \quad
 $T_3=$\pstree[levelsep=0.9cm, treesep=0.4cm]{\TC}{
\pstree{\TC}{\TC\TC\pstree{\TC}{\TC}\TC}
\pstree{\TC}{\TC}\pstree{\TC}{\TC\TC\pstree{\TC}{\TC}}
 }
 $T_4=$\pstree[levelsep=0.9cm, treesep=0.4cm]{\TC}{\TC\TC}
\end{center}
\end{example}

Clearly the process can be reversed, and to every forest with $m$
trees and with total number of $\l$ edges we can associate an
element of $\P^\l(\Omega_m)$.

%\bibliographystyle{alpha}
%\bibliography{bibcat}
\end{document}